\documentclass[11pt,a4paper]{amsart}
\usepackage[T1]{fontenc}
\usepackage{lmodern,mathrsfs,amsmath,amssymb,amsthm,mathtools,booktabs,array,longtable}
\usepackage[left=22mm,right=22mm,top=14mm,bottom=22mm,headheight=14pt]{geometry}
\newcommand{\href}[2]{#2}
\usepackage{microtype}
\newtheorem{theorem}{Theorem}[section]
\newtheorem{lemma}[theorem]{Lemma}
\newtheorem{proposition}[theorem]{Proposition}
\newtheorem{corollary}[theorem]{Corollary}
\theoremstyle{remark}
\numberwithin{equation}{section}
\DeclareMathOperator{\Pf}{Pf}
\DeclareMathOperator{\Res}{Res}
\DeclareMathOperator{\diag}{diag}
\DeclareMathOperator{\rank}{rank}
\newcommand{\Q}{\mathbb Q}
\newcommand{\G}{\mathsf G}
\newcommand{\U}{\mathsf U}
\newcommand{\B}{\mathsf B}
\newcommand{\J}{\mathsf J}
\newcommand{\CC}{\mathscr C}
\newcommand{\OO}{\mathscr O}
\newcommand{\HH}{\mathscr H}
\newcommand{\PP}{\mathscr P}
\newcommand{\EE}{\mathcal E}

\allowdisplaybreaks
\title[The Colomo--Pronko conjecture]{The Colomo--Pronko conjecture for frozen-corner alternating sign matrices}
\author{Yinjie Li}
\makeatletter
\patchcmd{\@maketitle}{\global\topskip42\p@}{\global\topskip10\p@}{}{\PackageError{CP}{Title spacing patch failed}{}}
\def\@setauthors{%
  \begin{center}
  \vspace{14pt}
  \normalsize YINJIE LI\textsuperscript{a}\\[2pt]
  \textsuperscript{a} Independent Researcher, Nanjing, China\\[2pt]
  Email: \texttt{stewarteur@gmail.com}
  \end{center}}
\makeatother
\date{}
\begin{document}
\begin{abstract}
We prove the Colomo--Pronko frozen-corner determinant conjecture for alternating sign matrices of every size and every freezing parameter. The proof connects a known multiple-integral formula for the enumeration with the Colomo--Pronko and Fischer--Reibnegger determinants. On the enumerative side, we obtain determinant representations built from polynomial kernels independent of the freezing parameter. An inverse identity for the commutator of a signed Pascal matrix with reversal identifies these kernels with the candidate matrix; in odd dimension, the comparison uses its one-dimensional nullspace and a central projection. The finite-dimensional algebraic core has been formalized in Lean~4. Combined with Colomo and Pronko's asymptotic theorem for the determinant, our result removes the conjectural input and yields an unconditional GUE Tracy--Widom limit for the diagonal intersection of the frozen boundary of a uniformly random alternating sign matrix.
\end{abstract}
\maketitle
\enlargethispage{2pt}

\section{Introduction}
Alternating sign matrices connect exact enumeration with the domain-wall six-vertex model. Their total number and the refinement recording the position of the unique $1$ in the first row are given by classical product formulas \cite{Z,Ku,Zref}. Operator formulas for monotone triangles \cite{Fischer} and multiply-refined enumeration \cite{Behrend} provide further ways to impose boundary data. Frozen-corner enumeration asks a different question: how many $n\times n$ matrices have an $s\times s$ square of zeros at a specified corner? Colomo and Pronko proposed a determinant for this count in the square-ice setting \cite{CP24}, and subsequently formulated the conjecture explicitly for ASMs \cite[Conjecture~1]{CP}. We prove it for all $n$ and $s$.

The probabilistic significance of this refinement comes from the frozen boundary. In a uniformly random large ASM, ordered regions occupy the four corners, while a disordered region remains between them. The limiting arctic curve was obtained by Colomo and Pronko \cite{CPshape}; a rigorous proof of boundary concentration, including the square domain, was given by Aggarwal \cite{Agg}. In the upper-left quadrant the curve is an arc of an ellipse, and its intersection with the diagonal is $(y_{\mathrm c},y_{\mathrm c})$, where
\[
 y_{\mathrm c}=1-\frac{\sqrt3}{2}.
\]
The frozen-corner count records fluctuations about this location. With the integer boundary coordinate $\xi_n$ defined in Section~2, the event $\{\xi_n>s\}$ is exactly that the top-left $s\times s$ square contains only zeros. Thus $B_{n,s}/A_n$ is its survival probability, as in \cite[(4.3)]{CP}.

An exact enumeration formula need not be suitable for extracting this critical-scale distribution. The known multiple integral has a number of variables that grows with the system, together with pairwise interaction factors. Colomo and Pronko's proposed determinant instead admits a Fredholm representation whose kernel can be analyzed near a double saddle point \cite[Sections~4.2--4.4]{CP24}. Their determinant asymptotics are stated in \cite[Theorem~4]{CP}: at
\[
 s=\left\lfloor ny_{\mathrm c}-\kappa n^{1/3}t\right\rfloor,
 \qquad \kappa=2^{-4/3}3^{-1/6},
\]
the determinant converges to the GUE Tracy--Widom distribution function $F_2(t)$ \cite{TW}. The missing input for the probabilistic conclusion was the identity between that determinant and $B_{n,s}/A_n$. Theorem~\ref{thm:main} supplies it. Corollary~\ref{cor:tw} below therefore gives the unconditional limit for $(ny_{\mathrm c}-\xi_n)/(\kappa n^{1/3})$.

This is a one-point fluctuation result at a fixed direction through the frozen boundary. It differs from the GOE Tracy--Widom theorem of Ayyer, Chhita and Johansson \cite{ACJ} for the maximum of the top path, where the observable is optimized over position. Numerical work of Pr\"ahofer and Spohn \cite{PS} had also supported the $n^{1/3}$ scale and the Tracy--Widom law for the ice-point boundary. Since uniform ASMs correspond to the six-vertex model at $\Delta=1/2$, rather than the free-fermion point $\Delta=0$, the corollary gives a GUE edge-fluctuation law in this non-free-fermionic model. It concerns the diagonal intersection, not convergence of the entire boundary to an Airy process.

The proof separates three structures. On the candidate side, we establish the equivalence with Fischer and Reibnegger's expression \cite[Conjecture~6]{FR} and reduce it to the binomial matrix $K_n$ of \eqref{eq:K}. On the enumerative side, we start independently from the known integral and derive Pfaffians of fixed polynomial kernels: the freezing parameter enters only through the central interval of retained coefficients. Reflection then converts these Pfaffians to ordinary determinants. The candidate reduction alone does not identify the resulting matrices.

The connection is provided by the inverse commutator identity
\[
 \Omega_{2N}^{-1}=\J_{2N}-C^{(N)}.
\]
Here $\Omega_{2N}=\G_{2N}R_{2N}-R_{2N}\G_{2N}$, with $\G_{2N}$ the signed Pascal matrix and $R_{2N}$ reversal, while $C^{(N)}$ is the coefficient matrix of the enumerative kernel. Polynomial divisibility and an exact coefficient truncation produce this inverse. In odd dimension, a one-dimensional nullspace replaces invertibility, and a projection permits removal of the central coordinate. The resulting two matrix identities identify the relevant minors for every freezing parameter. Their scalar normalization uses only unrestricted enumeration. Thus the proof links the integral and determinant descriptions through a finite-dimensional algebraic mechanism, rather than by further evaluation of special cases.

Section~\ref{sec:prelim} establishes the refined-polynomial identities. Section~\ref{sec:candidate} gives the binomial reduction. Sections~\ref{sec:integral} and \ref{sec:fold} evaluate the enumeration integral, and Sections~\ref{sec:inverse}--\ref{sec:connections} establish the matrix comparison. Section~\ref{sec:closure} completes the proof, including the boundary cases. The finite-dimensional algebraic core has been formalized in Lean~4; Section~\ref{sec:lean} specifies its scope. The probabilistic corollary uses the asymptotic theorem cited above, separately from the exact enumeration proof.

\section{The enumeration problem and the conjectural determinant}
An alternating sign matrix (ASM) is a square matrix with entries in $\{0,1,-1\}$, with alternating nonzero entries along each row and column, each of which sums to $1$. Let $B_{n,s}$ count $n\times n$ ASMs whose top-left $s\times s$ square is zero. We include $s=0$, meaning no restriction. Reflection shows that the choice of corner does not affect the count. Write
\begin{equation}\label{eq:asm}
 A_0=1,\qquad A_n=\prod_{j=0}^{n-1}\frac{(3j+1)!}{(n+j)!}.
\end{equation}
The ASM and refined ASM theorems \cite{Z,Ku,Zref} give the normalized refinement
\begin{equation}\label{eq:h}
 h_m(z)=\sum_{a=0}^{m-1}
 \frac{\binom{m+a-1}{m-1}\binom{2m-a-2}{m-1}}
 {\binom{3m-2}{m-1}}z^a,\qquad
 b_m=h_m(0)=\frac{A_{m-1}}{A_m}>0.
\end{equation}
In particular $h_m(1)=1$ and $h_m(z)=z^{m-1}h_m(1/z)$.

All auxiliary matrices below are indexed from zero. The interval $a:b$ means $a,a+1,\ldots,b-1$; binomial coefficients outside their natural range are zero. Empty determinants and Pfaffians are $1$, with the relevant boundary cases justified separately.
Put $r=n-s$, $m_i=r+i+1$, and, for $0\le i<s$, define
\begin{equation}\label{eq:cpM}
 f_i^\pm(z)=(1\pm(-1)^{i+1}z)(1-z)^iz^{-i-1}h_{m_i}(z),\qquad
 M_{ij}=b_{m_j}^{-1}\Res_{z=0}\Res_{w=0}
 \frac{f_i^+(z)f_j^-(w)}{1-z-w}.
\end{equation}
These are \cite[(1.7)--(1.9)]{CP}, after subtracting one from each matrix index. The normalization is attached to the column index $j$.

\begin{theorem}\label{thm:main}
For all $n\ge1$ and $0\le s\le n$,
\begin{equation}\label{eq:main}
 B_{n,s}=A_n\det(I_s-M).
\end{equation}
Both sides are zero when $2s>n$.
\end{theorem}
We shall use the matrix
\begin{equation}\label{eq:K}
 (K_n)_{ij}=\binom{i+j}{i}+(-1)^i\binom ji-(-1)^j\binom ij,
 \qquad E_{n,s}=\diag(0_{n-s},I_s).
\end{equation}

\subsection*{Diagonal boundary fluctuations}
For a uniformly sampled ASM $(a_{ij})_{1\le i,j\le n}$, define
\[
 \xi_n=\min\{\max(i,j):a_{ij}\ne0\}.
\]
Thus $\xi_n-1$ is the side length of the largest all-zero square at the top-left corner. This fixes the lattice convention for the diagonal boundary coordinate by the exact identity
\begin{equation}\label{eq:diagonal-tail}
 \mathbb P_n(\xi_n>s)=\frac{B_{n,s}}{A_n},\qquad 0\le s\le n,
\end{equation}
which is the survival-probability convention of \cite[(4.3)]{CP}. Let $F_2$ be the GUE Tracy--Widom distribution function in its standard Airy-kernel normalization \cite{TW}, and put $y_{\mathrm c}=1-\sqrt3/2$ and $\kappa=2^{-4/3}3^{-1/6}$.

\begin{corollary}[Diagonal Tracy--Widom fluctuations]\label{cor:tw}
For uniformly random $n\times n$ alternating sign matrices, and every $t\in\mathbb R$,
\begin{equation}\label{eq:tw-limit}
 \lim_{n\to\infty}\mathbb P_n\left(
 \frac{ny_{\mathrm c}-\xi_n}{\kappa n^{1/3}}\le t\right)=F_2(t).
\end{equation}
In particular, the diagonal boundary coordinate has $n^{1/3}$-scale fluctuations, or $n^{-2/3}$-scale fluctuations after division by $n$.
\end{corollary}
\begin{proof}
Set $s_n(t)=\lfloor ny_{\mathrm c}-\kappa n^{1/3}t\rfloor$. For all sufficiently large $n$, this is an admissible freezing parameter. By Theorem~\ref{thm:main} and \eqref{eq:diagonal-tail},
\[
 \mathbb P_n\left(\frac{ny_{\mathrm c}-\xi_n}{\kappa n^{1/3}}<t\right)
 =\frac{B_{n,s_n(t)}}{A_n}
 =\left.\det(I_s-M)\right|_{s=s_n(t)}.
\]
The first equality uses that $\xi_n$ is integer-valued. Colomo and Pronko's determinant limit \cite[Theorem~4, (4.4)]{CP}, derived in \cite[Section~4.4]{CP24}, gives $F_2(t)$. Continuity of $F_2$ also gives the limit with $\le t$, proving \eqref{eq:tw-limit}.
\end{proof}
The external asymptotic input in this corollary is the stated determinant limit. Its identification with a probability is now supplied by Theorem~\ref{thm:main}; the exact enumeration proof in Sections~3--10 does not use the asymptotic result.

\section{Refined polynomials and their identities}\label{sec:prelim}
\subsection{Fractional substitution and the coefficient gap}
\begin{lemma}\label{lem:fractional}
For $m\ge1$, with $\Phi_m(t)=h_m(t/(t-1))$,
\begin{equation}\label{eq:fractional}
 \Phi_m(t)=(1-t)^m h_m(t)
 +(-1)^{m-1}\frac{t^{2m-1}}{(1-t)^{m-1}}h_m(1-t).
\end{equation}
Consequently $\Phi_m(t)\equiv(1-t)^mh_m(t)\pmod{t^{2m-1}}$ in $\Q[[t]]$.
\end{lemma}
\begin{proof}
The case $m=1$ is immediate. Set $q=m-1\ge1$. Dividing the coefficients of $h_m$ by $b_m$ gives
$(-q)_a(q+1)_a/((-2q)_a a!)$, where $(c)_a$ denotes a rising factorial. Comparing consecutive coefficients proves
\[
 z(1-z)h_m''+(-2q-2z)h_m'+q(q+1)h_m=0.
\]
For $P(z)=(1-z)^{2q+1}h_m(z)=\sum p_kz^k$, substitution gives
\[
 z(1-z)P''+(-2q+4qz)P'-q(3q+1)P=0,
 \qquad (k+1)(k-2q)p_{k+1}=(k-q)(k-3q-1)p_k.
\]
At $k=q$ and then $q<k<2q$, this implies $p_{q+1}=\cdots=p_{2q}=0$. No division by $k-2q$ is made at $k=2q$.
Let $Q(z)=(1-z)^qh_m(z/(z-1))$, a polynomial of degree at most $q$. For $0\le k\le q$,
\[
 [z^k]Q=(-1)^k\sum_{a=0}^k[z^a]h_m\binom{q-a}{k-a}
 =b_m\frac{(-q)_k(-3q-1)_k}{(-2q)_k k!}.
\]
Here the finite Chu--Vandermonde evaluation follows by substituting the displayed hypergeometric coefficients and cancelling factorials; equivalently it is the coefficient form of $(1+z)^u(1+z)^v=(1+z)^{u+v}$, continued as a polynomial identity in $u,v$. The same expression follows for $p_k$ from its recurrence and $p_0=b_m$. Finally, reciprocity of $h_m$ gives $z^{3q+1}P(1/z)=-P(z)$. The coefficient gap therefore implies
$P(z)=Q(z)-z^{3q+1}Q(1/z)$. Substituting the definition of $Q$ proves \eqref{eq:fractional}.
\end{proof}
We will also need, for every integer $p$ and $d\ge0$,
\begin{equation}\label{eq:coeff-transform}
 [t^d](1-t)^{-p}h\left(\frac t{t-1}\right)
 =(-1)^d[z^d](1-z)^{d+p-1}h(z).
\end{equation}
It is enough to check $h(z)=z^a$: both sides are zero for $a>d$, and otherwise both equal $(-1)^a\binom{p+d-1}{d-a}$, with generalized binomial coefficients.

\subsection{The three-term recurrence}
Set $\beta_2=2$ and, for $m\ge3$,
\[
 \beta_m=\frac{b_{m-1}}{b_m}
 =\frac{3(3m-4)(3m-2)}{4(2m-3)(2m-1)}.
\]
\begin{lemma}\label{lem:recurrence}
For $m\ge3$,
\begin{equation}\label{eq:recurrence}
 \beta_m(z-1)^2h_m(z)=z^2h_{m-2}(z)
 +\frac{(z-2)(z+1)(2z-1)}2h_{m-1}(z).
\end{equation}
\end{lemma}
\begin{proof}
Write $p_m=h_m/b_m=\sum c_{m,a}z^a$, where
\[
 c_{m,a}=\frac{(m+a-1)!(2m-a-2)!}{a!(m-a-1)!(2m-2)!}\quad(0\le a<m),
\]
and extend $c_{m,a}$ by zero outside this range. The assertion becomes
\begin{equation}\label{eq:normalized-rec}
 (1-z)^2p_m=(1-\tfrac32z-\tfrac32z^2+z^3)p_{m-1}
 +\beta_{m-1}z^2p_{m-2}.
\end{equation}
For $m=3,4$ this follows on inserting $p_1=1$, $p_2=1+z$, $p_3=1+\frac32z+z^2$, and $p_4=1+2z+2z^2+z^3$. For $m\ge5$, both sides are reciprocal with respect to degree $m+1$. It suffices to compare $0\le a\le\lfloor(m+1)/2\rfloor$. Appendix~\ref{app:certificate} gives the seven coefficient ratios and their nonzero denominators; the full polynomial certificate is displayed in Supplementary Material, Section~S.1. That calculation includes $a=0,1,2$, where truncated coefficients vanish. It proves \eqref{eq:normalized-rec} for all coefficients.
\end{proof}

\subsection{An adjacent-order identity}
Put
\[
 a(z,w)=1-z+zw,\qquad \rho(z)=\frac{z-1}{z},\qquad
 \eta(z)=\frac z{z-1},\qquad
 \tau(z)=-\frac{(z-2)(z+1)(2z-1)}{2z(z-1)}.
\]
Then $\rho^3(z)=z$, $\tau(\rho z)=\tau(z)$, $\tau(\eta z)=-\tau(z)$, and
\begin{equation}\label{eq:tau-diff}
 \tau(w)-\tau(z)=-\frac{(w-z)a(z,w)a(w,z)}{zw(z-1)(w-1)}.
\end{equation}
These identities are obtained by clearing the displayed denominators and multiplying polynomials.
\begin{lemma}\label{lem:adjacent}
For $m\ge2$,
\begin{equation}\label{eq:adjacent}
 h_m(\rho z)h_{m-1}(z)+\frac{(z-1)^2}{z}h_m(z)h_{m-1}(\rho z)
 =b_m z^{m-2}(z^2-z+1).
\end{equation}
\end{lemma}
\begin{proof}
Set $g_m(z)=\rho(z)^m h_m(z)$. Equation~\eqref{eq:recurrence} is
$\beta_mg_m=g_{m-2}-\tau g_{m-1}$. Define
$W_m=g_m(z)g_{m-1}(\rho z)-g_{m-1}(z)g_m(\rho z)$.
Applying the recurrence at both arguments and using $\tau\rho=\tau$ yields
$\beta_m W_m=-W_{m-1}$. Its initial value is
$W_2=-(z^2-z+1)/(2z(z-1))$. Since $b_2=1/2$ and $\beta_m=b_{m-1}/b_m$,
\[
 W_m=(-1)^{m-1}b_m\frac{z^2-z+1}{z(z-1)}.
\]
On the other hand substitution of $g_m$ expresses $W_m$ as $(-1)^{m-1}/(z^{m-1}(z-1))$ times the left side of \eqref{eq:adjacent}. This proves the formula. In determinant form it says
\begin{equation}\label{eq:adj-det}
 \left.\det\begin{pmatrix}(z-1)h_m(z)&zh_{m-1}(z)\\
 (w-1)h_m(w)&wh_{m-1}(w)\end{pmatrix}\right|_{w=\rho z}
 =b_mz^{m-2}(z^2-z+1).
\end{equation}
\end{proof}

\section{Reduction to the binomial matrix}\label{sec:candidate}
\begin{proposition}\label{prop:candidate}
For $n\ge1$, $0\le s\le n$ and an indeterminate $\lambda$,
\begin{equation}\label{eq:candidate}
 A_n\det(I_s-\lambda M)=\det(K_n-\lambda E_{n,s}).
\end{equation}
\end{proposition}
\subsection{Triangular factors and the Fischer--Reibnegger expression}
All matrices in this subsection have size $n$. Set $S_{ij}=\delta_{i,j+1}$, $J_+=S^T$, $\Xi_{ii}=(-1)^{i+1}$ and $\mathcal R=(I-J_+)^{-1}$. Define
\[
 P_{ij}=\binom{i+j}{i},\quad B^F_{ij}=(-1)^{i+j}P_{ij},\quad
 (L_B)_{ij}=(-1)^{i+j}\binom ij,\quad U_B=L_B^T.
\]
Finite Vandermonde summation gives $B^F=L_BU_B$. Write $J_++B^F=L^FU^F$, with $L^F$ unit lower triangular; existence and nonzero pivots follow from the construction below. In the notation corresponding to \cite[Remark 5 and Conjecture 6]{FR}, put
\begin{align*}
 H^F&=\Xi U^F\Xi\mathcal R,&G^F&=\diag((H^F_{ii})^{-1})H^F,\\
 C^F&=I+(-1)^n\Xi L_BS,&D^F&=I+(-1)^{n+1}J_+U_B\Xi=2I-(C^F)^T.
\end{align*}
We claim
\begin{equation}\label{eq:Hinv}
 ((H^F)^{-1})_{ij}=[t^{j-i}]\Phi_{j+1}(t)\quad(i\le j),\qquad
 H^F_{ii}=b_{i+1}^{-1},\qquad \det H^F=A_n.
\end{equation}
Here and below an upper-triangular entry is zero if $i>j$.
To prove the claim, set $\mathcal T=P-J_+$ and $\mathcal W=\mathcal T\mathcal R$; then
$\mathcal W_{aj}=\binom{a+j+1}{a+1}-[j>a]$.
Let $Z_{ij}=[t^{j-i}]\Phi_{j+1}$. For $m\ge1$,
\[
 (\mathcal W Z)_{a,m-1}
 =[t^{m-1}]\frac{\Phi_m(t)}{(1-t)^{a+2}}
 -[t^{m-a-2}]\frac{\Phi_m(t)}{1-t}.
\]
For $a<m-1$, Lemma~\ref{lem:fractional}, \eqref{eq:coeff-transform} and reciprocity make these coefficients equal. For $a=m-1$ the expression is $h_m(1)=1$. Thus $\mathcal WZ$ is unit lower triangular. Since $Z$ has nonzero diagonal $b_m$, this constructs the claimed LU decomposition of $\mathcal T$, hence that of $J_++B^F=\Xi\mathcal T\Xi$, and proves \eqref{eq:Hinv}.

Let a subscript $*$ mean restriction to $r:n$, where $r=n-s$. Define lower and upper triangular $s\times s$ matrices
\begin{align*}
 \ell_{ij}&=b_{m_i}^{-1}[z^{i-j}](1-z)^{i-j-1}
 (1+(-1)^{i+1}z)h_{m_i}(z) &&(i\ge j),\\
 u_{ij}&=b_{m_j}^{-1}[z^{j-i}](1-z)^{j-i-1}
 (1+(-1)^jz)h_{m_j}(z) &&(i\le j).
\end{align*}
The identity
\[
 \sum_{k\ge0}\frac{z^kw^k}{(1-z)^{k+1}(1-w)^{k+1}}=\frac1{1-z-w}
\]
shows that $M=\diag(b_{m_i})\ell u$; the entry $(i,j)$ uses only $k\le\min(i,j)$. Let $\varepsilon_{ii}=(-1)^{i+1}$, $L_*=(G^F_*)^{-T}C^F_*$, and $U_*=D^F_*(G^F_*)^{-1}$. Their entries give
\begin{equation}\label{eq:parityLU}
 (L_*,U_*)=
 \begin{cases}(\varepsilon\ell\varepsilon,\varepsilon u\varepsilon),&s\text{ odd},\\
 (\varepsilon u^T\varepsilon,\varepsilon\ell^T\varepsilon),&s\text{ even}.
 \end{cases}
\end{equation}
With $d=i-j$ and $m=r+i+1$, the lower entry is
\[
 (L_*)_{ij}=b_m^{-1}\left([t^d]\Phi_m(t)
 +(-1)^{s+j}[t^{d-1}](1-t)^{-(m-d+1)}\Phi_m(t)\right),
\]
obtained from $(C^F_*)_{ij}=\delta_{ij}+(-1)^{s+j}\binom{r+i}{r+j+1}$ and \eqref{eq:Hinv}. Apply \eqref{eq:fractional} to the first term and \eqref{eq:coeff-transform} to the second. Here $d\le i\le m-1$, so the discarded term has strictly higher degree. The upper entries follow by transposing and using $D^F=2I-(C^F)^T$. This proves \eqref{eq:parityLU}, including its parity switch, and hence
\begin{equation}\label{eq:FRcandidate}
 A_n\det(I_s-\lambda M)=A_r\det((G^F_*)^TH^F_*-\lambda C^F_*D^F_*).
\end{equation}
Multiplication by $(G^F_*)^{-T}$ and $(G^F_*)^{-1}$ reduces the matrix to $\diag(b_{m_i}^{-1})-\lambda L_*U_*$. Equation~\eqref{eq:parityLU}, determinant invariance under transpose, and $\det(I-AB)=\det(I-BA)$ give the asserted expression. This establishes the equivalence with \cite{FR}, including the deformation parameter.

\subsection{Eliminating the triangular factors}
Put $Q_+=I-J_++J_+^2$. The matrix $\mathcal T$ has generating function
\[
 \frac{1-y+y^2}{(1-x-y)(1-xy)}.
\]
Right multiplication by $Q_+^{-1}$ removes the numerator in this finite upper-triangular coefficient calculation, giving a symmetric matrix. Uniqueness of the triangular decomposition of a symmetric matrix, together with \eqref{eq:Hinv}, gives
\[
 G^F=\widetilde L^TQ_+\mathcal R,
 \qquad \widetilde L=\Xi L^F\Xi.
\]
Consequently $V=(G^F)^TH^F=(H^F)^T\diag(b_i)_{i=1}^nH^F$ has coefficient generating function
\begin{equation}\label{eq:Vgf}
 \frac{(1-x+x^2)(1-y+y^2)}{(1-x)(1-y)(1-x-y)(1-xy)}.
\end{equation}
Its leading $r\times r$ determinant is $A_r$, and its Schur complement is $(G^F_*)^TH^F_*$. Equation~\eqref{eq:FRcandidate} becomes
$\det(V-\lambda\Pi C^F_*D^F_*\Pi^T)$, where $\Pi$ embeds the last $s$ coordinates.

To eliminate $V$, let $F_{ab}=(-1)^b\binom a{b+1}$, a strictly lower triangular matrix. On coefficient vectors its generating-series action is
\[
 (Ff)(x)=\frac{x}{(1-x)^2}f\left(\frac{x}{x-1}\right).
\]
Write $q(x)=x^2/(1-x)$, $g=1+q$, and let $L_g$ denote multiplication by $g$, truncated at degree $n$. Substitution gives $F^2=-L_q$ and $FL_g=L_gF$. If $W$ has generating function $((1-x-y)(1-xy))^{-1}$, then $V=L_gWL_g^T$. For $\sigma=\pm1$ set $C_\sigma=I+\sigma F$, $D_\sigma=I-\sigma F^T$. These are unit triangular, and
\[
 C_\sigma^{-1}VD_\sigma^{-1}=(I-\sigma F)W(I+\sigma F)^T.
\]
The three terms needed to evaluate the right side have generating functions
\[
 F_xW=\frac{x}{(1-x+xy)(1-y+xy)},\quad
 F_yW=\frac{y}{(1-x+xy)(1-y+xy)},\quad F_xF_yW=xyW.
\]
Thus the result has generating function
\[
 \frac1{1-x-y}+\sigma\left(\frac1{1-y+xy}-\frac1{1-x+xy}\right),
\]
which gives $K_n$ for $\sigma=1$ and $K_n^T$ for $\sigma=-1$. In each of these products the triangular factors ensure that an entry of index less than $n$ uses only indices less than $n$, so the calculations hold for the stated matrix size.
For $\sigma=(-1)^n$, $C_\sigma=C^F$, $D_\sigma=D^F$, and
$C_\sigma E_{n,s}D_\sigma=\Pi C^F_*D^F_*\Pi^T$. This proves Proposition~\ref{prop:candidate}, and at $\lambda=0$ gives
\begin{equation}\label{eq:detK}\det K_n=A_n.\end{equation}

\subsection{The forbidden range and the corner convention}
If the top-left square is frozen, the $s$ positions in row $s$ of the monotone triangle all lie among $s+1,\ldots,n$. Such a strictly increasing row requires $s\le n-s$, so $B_{n,s}=0$ for $2s>n$. Equivalently, after reflecting the ASM vertically, row $n-s$ must begin with $1,\ldots,s$.
For the candidate let $T_{ij}=(-1)^j\binom ij$. Finite binomial inversion gives $T^2=I$, $TT^T=(\binom{i+j}{i})$, and
\[
 K_n-I=(T+I)(T^T-I).
\]
Since $T$ is triangular and an involution in characteristic zero, $\rank(T^T-I)=\lfloor n/2\rfloor$. Hence
$\rank(K_n-E_{n,s})\le\lfloor n/2\rfloor+n-s<n$ when $2s>n$.

For the bottom-left convention, the equivalent monotone-triangle condition fixes the entries $1,\ldots,s$ in row $n-s$ and in every row below it, including the original bottom row. When deleting the last $s$ entries of the corresponding northeast diagonals one must retain the fixed new bottom boundary. A description fixing only the original bottom $s$ entries of a diagonal misses this condition. The zero-square definition used here agrees with \cite[Conjecture 6]{FR}.

\section{From frozen-corner enumeration to Pfaffians}\label{sec:integral}
\subsection{The known integral and its normalization}
Suppose $2s\le n$, and set $r=n-s$, $d=n-2s$. Define
\[
 F_k^{r,d}(z)=(z-1)^{d-1-k}z^kh_{r-k}(z)\quad(0\le k<d),\qquad
 \Delta_{r,d}(z)=\det[F_k^{r,d}(z_i)]_{i,k=0}^{d-1},
\]
and $h_{r,d}(z)=\Delta_{r,d}(z)/\prod_{i<j}(z_i-z_j)$. The determinant is the transpose of the numerator in \cite[(2.8)]{CP24}; the denominator has the same, descending, orientation.
The external enumeration input is the second multiple-integral representation \cite[(2.9)]{CP24}, which is a proved formula, distinct from the conjectural determinant of that paper. At $t=1$, $\Delta=1/2$ in its physical notation, take its lattice size to be $n$, its $r$ to be $n-s$, and its $s$ to be our $s$. Its two polynomial orders $n-s$ and $r$ coincide, and the number of integrations is $r-s=d$. Since all square-ice configurations have equal weight, multiplying the probability by $A_n$ gives
\begin{equation}\label{eq:true-integral}
 B_{n,s}=\frac{A_r^2}{d!}\Res_{z_1=1}\cdots\Res_{z_d=1}
 \prod_i\frac1{z_i-1}
 \prod_{i\ne j}\frac{z_i-z_j}{a(z_i,z_j)}
 h_{r,d}(z)h_{r,d}(\eta z).
\end{equation}
In detail, the prefactor before simplification is
\[
 \frac{A_n}{d!}\frac{\prod_{j=n-s+1}^n b_j}{\prod_{j=1}^r b_j}
 =\frac{A_n}{d!}\frac{A_r/A_n}{1/A_r}=\frac{A_r^2}{d!}.
\]
All residues below are taken on fixed, small positively oriented circles about $1$, in a common neighbourhood where the paired denominators are nonzero since $a(1,1)=1$. Substitution of $\eta$ produces only finite-order single-variable poles. The residues are therefore ordinary Laurent coefficients and may be evaluated in any order.

\subsection{Middle columns and the Pfaffian identity}
Equation~\eqref{eq:recurrence} gives
\begin{equation}\label{eq:middle}
 F_{k+2}^{r,d}=\tau F_{k+1}^{r,d}+\beta_{r-k}F_k^{r,d}
 \quad(0\le k\le d-3).
\end{equation}
The smallest recurrence order is $s+3\ge3$. If $d=2m\ge2$, choose $A=F_{m-1}^{r,d}$, $B=F_m^{r,d}$. Then
\begin{equation}\label{eq:even-columns}
 \Delta_{r,2m}=\gamma_e\det[A_i\tau_i^j\mid B_i\tau_i^j]_{j=0}^{m-1},
 \qquad \gamma_e=\prod_{a=1}^{m-1}\beta_{r-m+a+1}^{-(m-a)}.
\end{equation}
The $j$th backward step in \eqref{eq:middle} has leading term
$(-1)^j\tau^jA/\prod_{a=1}^j\beta_{r-m+a+1}$, while forward steps have leading term $\tau^jB$. Elimination of lower powers is triangular. Reversing the left group of columns cancels the product of the backward signs.
For $d=2m+1$, use $A=F_m^{r,d}$, $B=F_{m+1}^{r,d}$ and obtain
\begin{equation}\label{eq:odd-columns}
 \Delta_{r,2m+1}=\gamma_o
 \det[(A_i\tau_i^j)_{j=0}^m\mid(B_i\tau_i^j)_{j=0}^{m-1}],
 \quad \gamma_o=\prod_{a=1}^m\beta_{r-m+a}^{-(m+1-a)}.
\end{equation}
For $m=0$ there is only the column $A$.
For physical sizes $2N$ and $2N+1$, respectively, $r=N+m$ and $r=N+m+1$. Put $b=b_{N+1}$. Telescoping the ratios $\beta_j=b_{j-1}/b_j$ gives
\begin{equation}\label{eq:gamma}
 A_r\gamma_e=A_Nb^{-m},\qquad A_r\gamma_o=A_Nb^{-m-1}.
\end{equation}
For example, the even product is $b_{N+1}^{-(m-1)}\prod_{j=N+2}^{N+m}b_j$; the odd product is $b_{N+1}^{-m}\prod_{j=N+2}^{N+m+1}b_j$. Substitution of these products gives both equalities.

We use the convention
$\Pf Z=(2^mm!)^{-1}\sum_{\pi\in S_{2m}}\operatorname{sgn}(\pi)\prod_{j=0}^{m-1}Z_{\pi(2j),\pi(2j+1)}$.
For independent $t_i,A_i,B_i,C_i,D_i$ set
\[
 H_{ij}=\frac{(A_iB_j-A_jB_i)(C_iD_j-C_jD_i)}{t_j-t_i}.
\]
The following is a homogeneous version of \cite[(1.8)]{IOTZ}:
\begin{equation}\label{eq:pf-product}
 \Pf H=\frac{\det[A_it_i^j\mid B_it_i^j]_{j=0}^{m-1}
 \det[C_it_i^j\mid D_it_i^j]_{j=0}^{m-1}}{\prod_{i<j}(t_j-t_i)}.
\end{equation}
We prove this identity by induction on $m$. As functions of $t_0$, both sides have only simple poles at the other $t_i$. At $t_0=t_1=t$, replace powers by $(t_i-t)^j$. Expanding along the first two rows leaves the two alternating factors, a factor $\prod_{i\ge2}(t_i-t)$ in each determinant, and the determinants for $2m-2$ points. The squared factors cancel those in the Vandermonde; the two column signs cancel. The residues therefore agree with the Pfaffian expansion by induction. Both sides are $O(t_0^{-1})$ at infinity (degrees $2m-2$ and $2m-1$ on the right). Their difference is zero. The case $m=1$ starts the induction.

The odd version is
\begin{equation}\label{eq:pf-border}
 \frac{\det[(A_it_i^j)_{j=0}^m\mid(B_it_i^j)_{j=0}^{m-1}]
 \det[(C_it_i^j)_{j=0}^m\mid(D_it_i^j)_{j=0}^{m-1}]}
 {\prod_{i<j}(t_j-t_i)}
 =\Pf\begin{pmatrix}H&(A_iC_i)_i\\-(A_iC_i)_i^T&0\end{pmatrix}.
\end{equation}
To obtain it, add a point $t_{2m+1}=T$ with $A=C=0$, $B=D=1$ in the even formula, multiply by $T$, and compare leading terms as $T\to\infty$. The new border is $A_iC_i$ and the signs of the two determinant expansions cancel. These are identities of rational functions; after clearing denominators they allow specialization to the data in \eqref{eq:even-columns}--\eqref{eq:odd-columns}.

Let $V(z)=\prod_{i<j}(z_j-z_i)$ and $q_d=d(d-1)/2$. We have
\[
 V(\eta z)=(-1)^{q_d}V(z)\prod_i(z_i-1)^{-(d-1)},\quad
 V(\tau z)=(-1)^{q_d}V(z)
 \frac{\prod_{i<j}a(z_i,z_j)a(z_j,z_i)}{\prod_i[z_i(z_i-1)]^{d-1}}.
\]
Substitution into \eqref{eq:true-integral} gives
\begin{equation}\label{eq:integral-pf-ready}
 B_{n,s}=\frac{(-1)^{q_d}A_r^2}{d!}\Res_{z_i=1}
 \frac{\Delta_{r,d}(z)\Delta_{r,d}(\eta z)V(z)}{V(\tau z)}
 \prod_i\frac1{z_i^{d-1}(z_i-1)}.
\end{equation}
The second determinant uses $\tau(\eta z)=-\tau(z)$. Its column powers contribute sign $(-1)^{m(m-1)}=1$ in even dimension and $(-1)^{m^2}=(-1)^m$ in odd dimension. The cleared polynomial identity extends the formula to coincident spectral values.

\subsection{The finite kernels and cancellation of paired poles}
Fix $N\ge1$ and $b=b_{N+1}$. Define the polynomials
\begin{equation}\label{eq:uv}
 \begin{split}
 u(x)&=xh_{N+1}(1+x),\qquad v(x)=(1+x)h_N(1+x),\\
 u^*(x)&=x^{N+1}u(1/x)=x^Nh_{N+1}((1+x)/x),\\
 v^*(x)&=x^{N+1}v(1/x)=(1+x)x^Nh_N((1+x)/x).
 \end{split}
\end{equation}
Their degree bounds are $N+1,N,N,N+1$, respectively. Put
\begin{align}
 W(x,y)&=u(x)v(y)-u(y)v(x),&W^*(x,y)&=u^*(x)v^*(y)-u^*(y)v^*(x),\nonumber\\
 \HH_N(x,y)&=\frac{W(x,y)W^*(x,y)}{b^2(y-x)(1+x+xy)(1+y+xy)},\nonumber\\
 \CC_N(x,y)&=\HH_N(x,y)-\frac{x^{2N}}{1+y+xy}+\frac{y^{2N}}{1+x+xy}.\label{eq:Ckernel}
\end{align}
\begin{lemma}\label{lem:ker1}
The function $\CC_N$ is a polynomial of degree at most $2N-1$ in each variable. If $\CC_N=\sum_{i,j<2N}c_{ij}x^iy^j$ and $C^{(N)}=(c_{ij})$, then
\begin{equation}\label{eq:Creflection}
 (C^{(N)})^T=-C^{(N)},\qquad R_{2N}C^{(N)}R_{2N}=-C^{(N)},
\end{equation}
where $R_q$ reverses the $q$ coordinates.
\end{lemma}
\begin{proof}
Both alternating factors are divisible by $y-x$, which cancels the diagonal denominator. Write $z=1+x$, $w=1+y$. On $1+y+xy=0$, equivalently $w=\rho(z)$, \eqref{eq:adj-det} gives
\[
 W=bz^{N-1}(z^2-z+1),\qquad
 W^*=-b\frac{(z-1)^{2N}}{z^{N+1}}(z^2-z+1).
\]
For the second equality use $\eta(w)=1-z$, $\eta(z)=\rho(1-z)$ in the same identity and interchange its rows. Since
$w-z=-(z^2-z+1)/z$ and $a(w,z)=(z^2-z+1)/z$, it follows that
\begin{equation}\label{eq:pair-cancel}
 \left.(1+y+xy)\HH_N\right|_{1+y+xy=0}=x^{2N}.
\end{equation}
Interchanging $x,y$ cancels the other paired denominator. The numerator of \eqref{eq:Ckernel}, after a common denominator is taken, is thus divisible separately by $y-x$, $1+x+xy$, and $1+y+xy$. These are pairwise relatively prime irreducibles in the unique factorization domain $\Q[x,y]$, so their product divides the numerator, including at intersections of the divisors. The numerator has degree at most $2N+2$ in each variable and the denominator has degree $3$ in each; hence the bound. Interchange of variables gives skew symmetry. The reciprocal definitions in \eqref{eq:uv}, applied also to both correction terms, give
$\CC_N(1/x,1/y)=-(xy)^{-(2N-1)}\CC_N(x,y)$, proving the reversal identity.
\end{proof}
For odd dimension define
\begin{equation}\label{eq:odd-kernel}
 \OO_N=(1+x)(1+y)\CC_N+x^{2N}-y^{2N}=\sum_{i,j=0}^{2N}o_{ij}x^iy^j,
 \qquad \mathscr V_N(x)=\frac{h_{N+1}(1+x)u^*(x)}{b^2}=\sum_{i=0}^{2N}\nu_ix^i.
\end{equation}
Then $O^{(N)}=(o_{ij})$ is skew symmetric and changes sign under reversal, while $\nu_{2N-i}=\nu_i$. Both factors defining $\mathscr V_N$ have strictly positive coefficients in degrees $0$ through $N$; thus every $\nu_i$ is positive. In particular $\alpha:=\nu_N>0$.
The correction $x^{2N}-y^{2N}$ gives the identity
\begin{equation}\label{eq:odd-tail}
 (1+x)(1+y)\HH_N=\OO_N+\frac{x^{2N+1}}{1+y+xy}-\frac{y^{2N+1}}{1+x+xy},
\end{equation}
since $(1+x)(1+y)=(1+y+xy)+x=(1+x+xy)+y$.

\subsection{Residue moments}
For a linear functional $\mathcal L$ and a skew kernel $H$, expansion of the determinant and Pfaffian proves
\begin{equation}\label{eq:pf-integral}
 \frac1{(2m)!}\mathcal L^{\otimes2m}
 \left(\det[p_i(z_j)]_{i,j=0}^{2m-1}\Pf[H(z_i,z_j)]\right)
 =\Pf[\mathcal L_z\mathcal L_w(p_i(z)H(z,w)p_j(w))].
\end{equation}
Relabeling the variables makes every determinant permutation contribute the same quantity: its sign cancels the Pfaffian permutation sign, accounting for $(2m)!$. Expansion of the remaining Pfaffian gives the right side with its factor $2^mm!$. For $2m+1$ variables with a border $g(z_i)$, the same proof gives the border $\mathcal L(p_i g)$.

Apply this identity to \eqref{eq:integral-pf-ready}, with $p_i(z)=z^i$. In even dimension the common factors of the middle columns at $z=1+x$ are $x^{m-1}z^{m-1}$, and at $\eta z$ they are $z^{m-1}x^{-(N+2m-1)}$. Their product is $z^{d-2}x^{-r}$. After the weight in \eqref{eq:integral-pf-ready} and the spectral difference \eqref{eq:tau-diff} are included, the two-variable moment kernel is
\[
 -b^2\frac{\HH_N(x,y)}{x^ry^r}.
\]
Replacing $z^i=(1+x)^i$ by $x^i$ is a unit triangular change of the moment basis. Its indices are $r-1-i$, $0\le i<2m$, hence the increasing central interval $s:2N-s$ after reversal. The two tails in \eqref{eq:Ckernel} cannot contribute, because all indices are below $2N$ and their denominators are units at $(0,0)$.
The odd two-variable moment kernel is
$-b^2(1+x)(1+y)\HH_N(x,y)/(x^ry^r)$ and the border is $b^2\mathscr V_N(x)/x^r$. Equation~\eqref{eq:odd-tail} shows that its higher tails cannot contribute to indices below $2N+1$.

Put $m=N-s$. The factorial $d!$ cancels in \eqref{eq:pf-integral}. Combining the integral prefactor, the column factors, the signs of the second determinant's powers, the moment factors, and the reversal sign, in that order, gives
\[
\begin{aligned}
 2N:\quad&((-1)^mA_r^2)\gamma_e^2\cdot1\cdot((-1)^mb^{2m})(-1)^m
       =(-1)^mA_N^2,\\
 2N+1:\quad&((-1)^mA_r^2)\gamma_o^2(-1)^m((-1)^mb^{2m+2})(-1)^m
       =A_N^2,
\end{aligned}
\]
where the final equalities use \eqref{eq:gamma}. Hence
\begin{align}
 B_{2N,s}&=(-1)^{N-s}A_N^2\Pf(C^{(N)})_{s:2N-s,s:2N-s},\label{eq:even-pf}\\
 B_{2N+1,s}&=A_N^2\Pf\begin{pmatrix}
 (O^{(N)})_{s:2N+1-s,s:2N+1-s}&(\nu_i)_{i=s}^{2N-s}\\
 -(\nu_i)_{i=s}^{2N-s}{}^T&0
 \end{pmatrix}.\label{eq:odd-pf}
\end{align}
The even $m=0$ case comes directly from the zero-variable formula \eqref{eq:true-integral}; the odd $m=0$ case is a one-variable residue and equals $A_N^2\nu_N$.

\section{Reflection and determinant formulas for the count}\label{sec:fold}
In reflection-paired coordinates (left half followed by reversed right half), a skew matrix changing sign under reversal has the form
$Z=\begin{pmatrix}A&B\\-B&-A\end{pmatrix}$, with $A^T=-A$, $B^T=B$. For $S=\begin{pmatrix}I&I\\I&-I\end{pmatrix}$,
\[
 S^TZS=\begin{pmatrix}0&2(A-B)\\2(A+B)&0\end{pmatrix},\qquad \det S=(-2)^m.
\]
Use $\Pf(S^TZS)=\det S\Pf Z$ and $\Pf\begin{pmatrix}0&T\\-T^T&0\end{pmatrix}=(-1)^{m(m-1)/2}\det T$. Restoring the original right-half order introduces the same reversal sign $(-1)^{m(m-1)/2}$. Thus the Pfaffian in original increasing order is $\det(A+B)$, and $(-1)^m\Pf Z=\det(-A-B)$ in that order.

In the odd bordered case the paired order is left, reversed right, center, border. Write its noncentral blocks as $A,B$, its left-to-center column as $h$, and its border as $(v,v,\alpha)^T$. Its Pfaffian in original increasing order is
\begin{equation}\label{eq:odd-fold}
 \det\begin{pmatrix}-A-B&h\\-2v^T&\alpha\end{pmatrix}.
\end{equation}
Move the center past $m$ right indices and eliminate the final block $\begin{pmatrix}0&\alpha\\-\alpha&0\end{pmatrix}$. The noncentral skew matrix becomes
$\bar Z+(\bar\nu\bar h^T-\bar h\bar\nu^T)/\alpha$, where $\bar h=(h,-h)^T$. Its folded block is $A+B-2hv^T/\alpha$. The center permutation and the even folding signs cancel, and multiplication by $\alpha$ gives \eqref{eq:odd-fold}.
Define, for $0\le i,j<N$,
\begin{equation}\label{eq:true-matrices}
 (D_N)_{ij}=-c_{ij}-c_{i,2N-1-j},\qquad
 \EE_N=\begin{pmatrix}
 (-o_{ij}-o_{i,2N-j})_{i,j<N}&(o_{iN})_{i<N}\\
 (-2\nu_j)_{j<N}&\nu_N
 \end{pmatrix}.
\end{equation}
\begin{theorem}\label{thm:true}
For $N\ge1$ and $0\le s\le N$,
\begin{equation}\label{eq:true-dets}
 B_{2N,s}=A_N^2\det(D_N)_{s:N,s:N},\qquad
 B_{2N+1,s}=A_N^2\det(\EE_N)_{s:N+1,s:N+1}.
\end{equation}
\end{theorem}
\begin{proof}
Apply the two folding calculations to the central intervals in \eqref{eq:even-pf} and \eqref{eq:odd-pf}. Restriction to a central interval commutes with pairing the reflected indices. In odd size the center and the final border are retained. This gives the tail matrices in \eqref{eq:true-matrices}.
\end{proof}
Theorem~\ref{thm:true} follows from the enumeration integral and the refined-polynomial and Pfaffian identities, independently of the candidate reduction and the matrix comparison that follows.

\section{The inverse commutator identity}\label{sec:inverse}
\subsection{Finite binomial operators}
For each positive integer $q$ define
\[
 (T_q)_{ij}=(-1)^j\binom ij,\quad
 (\G_q)_{ij}=(-1)^{i+j}\binom{i+j}{i},\quad
 (\U_q)_{ij}=(-1)^{j-i}\binom{q-1-i}{j-i},\quad \B_q=R_q\U_q.
\]
Binomial inversion and Vandermonde summation give
\begin{equation}\label{eq:binomial}
 T_q^2=I,\quad \B_qT_q=R_q\B_q,\quad
 \B_q\B_q^T=\G_q,\quad \det\G_q=1.
\end{equation}
For the middle identity, the required alternating sum is
\[
 \sum_{l=0}^i(-1)^l\binom il\binom{q-1-i+l}{j}
 =[t^j](1+t)^{q-1-i}(1-(1+t))^i
 =(-1)^i\binom{q-1-i}{j-i}.
\]
For the Gram identity the sum is
$(-1)^{i+j}\sum_k\binom i{q-1-k}\binom j{q-1-k}
=(-1)^{i+j}\binom{i+j}i$; the determinant follows from the unit diagonal of $\U_q$. Binomial inversion itself follows from
$\binom ik\binom kj=\binom ij\binom{i-j}{k-j}$ and the alternating binomial sum. Applying \eqref{eq:binomial} to $K-I=(T+I)(T^T-I)$ yields
\begin{equation}\label{eq:K-congruence}
 \B_q(K_q-I)\B_q^T=(R_q+I)\G_q(R_q-I).
\end{equation}

If $f$ has degree less than $q$, multiplication by $\G_q$ on its coefficient column is the finite operator
\begin{equation}\label{eq:G-action}
 f\longmapsto\left[\frac1{1+x}f\left(-\frac1{1+x}\right)\right]_{<q}.
\end{equation}
On $f=x^j$, the coefficient of $x^i$ is $(-1)^{i+j}\binom{i+j}i$. Moreover
\begin{equation}\label{eq:GRG}
 (\G_qR_q\G_q^{-1})_{ij}=
 \begin{cases}(-1)^{q-1+i}\binom{q+i}{i-j},&i\ge j,\\0,&i<j.\end{cases}
\end{equation}
Writing $F$ for the displayed lower triangular matrix, we compute $F\G_q$. The needed sum is
\[
 \sum_{k=0}^i(-1)^k\binom{q+i}{i-k}\binom{k+j}k
 =[t^i](1+t)^{q+i}(1+t)^{-j-1}
 =\binom{q+i-j-1}i.
\]
Its exterior sign $(-1)^{q-1+i+j}$ agrees with $(\G_qR_q)_{ij}$. Since $\G_q$ is invertible, this proves \eqref{eq:GRG}. For even $q$ the coefficient kernel is therefore
\begin{equation}\label{eq:rational-trunc}
 -\left[\frac1{(1+x)^q(1+x+xy)}\right]_{x<q}.
\end{equation}
Expand $(1+x+xy)^{-1}=\sum_{j\ge0}(-xy)^j/(1+x)^{j+1}$; the $x$ truncation also bounds the $y$ degree by $q-1$.
Define
\begin{equation}\label{eq:omega}
 \Omega_q=\G_qR_q-R_q\G_q,\qquad
 \J_q=R_q\G_q^{-1}-\G_q^{-1}R_q.
\end{equation}
Both are skew symmetric and change sign under reversal. We shall prove the invertibility of $\Omega_q$ for even $q$ below.

\subsection{Univariate kernel identities}
Take $q=2N$, $\epsilon=(-1)^{N+1}$, $\theta(x)=-1-x$, and $\sigma(x)=-1/(1+x)$. Lemma~\ref{lem:fractional}, at $t=1+x$, gives
\begin{equation}\label{eq:theta}
 u^*=\epsilon x^qu-(1+x)^qu(\theta x),\qquad
 v^*=-\epsilon x^qv-(1+x)^qv(\theta x).
\end{equation}
For the first formula multiply that lemma for $h_{N+1}$ by $x^N$, and for the second multiply its version for $h_N$ by $(1+x)x^N$, using reciprocity. Lemma~\ref{lem:adjacent} at $z=1+x$ gives
\begin{equation}\label{eq:uvpair}
 uv^*+u^*v=b(1+x)^q(x^2+x+1).
\end{equation}
Before multiplication, the left side of the adjacent-order identity is
\[
 \left(\frac{x}{1+x}\right)^N
 \left[h_{N+1}((1+x)/x)h_N(1+x)
 +xh_{N+1}(1+x)h_N((1+x)/x)\right].
\]
The left side of \eqref{eq:uvpair} is the bracket times $(1+x)x^N$, a relative factor $(1+x)^{N+1}$. This gives the scalar $b$.
The reciprocal polynomials also satisfy the four rational identities
\begin{equation}\label{eq:sigma}
 \begin{aligned}
 u(\sigma x)&=-\frac{u^*(x)}{(1+x)^{N+1}},&
 v(\sigma x)&=\frac{v^*(x)}{(1+x)^{N+1}},\\
 u^*(\sigma x)&=\frac{\epsilon u(\theta x)}{(1+x)^{N+1}},&
 v^*(\sigma x)&=\frac{\epsilon v(\theta x)}{(1+x)^{N+1}}.
 \end{aligned}
\end{equation}
For the first two, substitute $1+\sigma x=x/(1+x)$ and apply reciprocity of $h_{N+1},h_N$. For the last two, substitute $\sigma x$ in the finite reversed expressions in \eqref{eq:uv}; their argument $(1+\sigma x)/(\sigma x)$ is $-x$. The signs from $(\sigma x)^N$ give the two factors $\epsilon$. Hence all four identities hold after clearing denominators, including specialization wherever the cleared expressions are defined.

\subsection{A polynomial with a prescribed reflection defect}
Let
\[
 \mathscr D=(y-x)(1+x+xy)(1+y+xy),\qquad A(x)=u(\theta x),\quad B(x)=v(\theta x),
\]
and set
\begin{equation}\label{eq:Pkernel}
 \mathscr F(x,y)=-\frac{\epsilon[A(x)v(y)+B(x)u(y)]
 [A(x)v^*(y)-B(x)u^*(y)]}{b^2\mathscr D(x,y)},\qquad
 \PP(x,y)=\mathscr F(x,y)+\frac{y^q}{x-y}.
\end{equation}
\begin{lemma}\label{lem:Pkernel}
The function $\PP$ is a polynomial of degree at most $q-1$ in each variable, and
\begin{equation}\label{eq:P-reflection}
 \PP(x,y)+(xy)^{q-1}\PP(1/x,1/y)
 =-\sum_{j=0}^{q-1}x^{q-1-j}y^j.
\end{equation}
Thus its actual coefficient matrix $P$ satisfies $P+R_qPR_q=-R_q$.
\end{lemma}
\begin{proof}
At $x=-1/(1+y)$, reciprocity gives $A(x)=-u(y)/(1+y)^{N+1}$ and $B(x)=v(y)/(1+y)^{N+1}$, so the first bracket vanishes. At $x=-(1+y)/y$ we have $\theta x=1/y$, so the second bracket vanishes. At $x=y$, \eqref{eq:theta} and \eqref{eq:uvpair} make the two brackets respectively $-b(y^2+y+1)$ and $\epsilon by^q(y^2+y+1)$. Hence $\mathscr F$ has residue $-y^q$ with respect to $x$ along the diagonal, canceled by the second term in \eqref{eq:Pkernel}. Clearing denominators and using the same three pairwise relatively prime factors as in Lemma~\ref{lem:ker1} proves joint divisibility. Numerator degrees are at most $q+2$, so the quotient has the claimed bounds.
Reciprocity gives
\[
 A(1/x)=-\epsilon x^{-N-1}A(x),\qquad
 B(1/x)=\epsilon x^{-N-1}B(x),
\]
while $u(1/x)=x^{-N-1}u^*(x)$ and similarly for $v$. Substitution in \eqref{eq:Pkernel} gives
\[(xy)^{q-1}\mathscr F(1/x,1/y)=-\mathscr F(x,y).\] The remaining sum is $(y^q-x^q)/(x-y)$, proving \eqref{eq:P-reflection}.
\end{proof}

\subsection{Exact truncation and a right inverse}
Write $W_{\theta,*}=u(\theta x)v^*(y)-u^*(y)v(\theta x)$. Equations~\eqref{eq:theta} and \eqref{eq:sigma} give
\begin{equation}\label{eq:exact-descent}
 \frac1{1+x}\HH_N(\sigma x,y)=\mathscr F(x,y)
 +\frac{x^qW(x,y)W_{\theta,*}(x,y)}{(1+x)^qb^2\mathscr D(x,y)}.
\end{equation}
Here $\mathscr D(\sigma x,y)=-\mathscr D(x,y)/(1+x)^3$, and
\[
 u^*(x)v(y)+u(y)v^*(x)=\epsilon x^qW(x,y)
 -(1+x)^q[A(x)v(y)+B(x)u(y)].
\]
The transformed first alternant is the negative of this expression divided by $(1+x)^{N+1}$; the second is $\epsilon W_{\theta,*}/(1+x)^{N+1}$. Combining these factors proves \eqref{eq:exact-descent}.

This step is interpreted in $\Q(y)[[x]]$. The constant term $\mathscr D(0,y)=y(1+y)$ is a nonzero element of the coefficient field. Therefore the last term of \eqref{eq:exact-descent} has zero coefficient of $x^i$ for every $i<q$. There is no expansion of $1/(x-y)$ at a common two-variable origin. The correction terms in \eqref{eq:Ckernel}, under the same substitution and the factor $(1+x)^{-1}$, become respectively
\[
 -\frac1{(1+x)^q(1+x+xy)},\qquad \frac{y^q}{x-y}.
\]
Using \eqref{eq:G-action}, \eqref{eq:rational-trunc}, and \eqref{eq:Pkernel}, we obtain an equality of finite coefficient matrices:
\begin{equation}\label{eq:GC}
 \G_qC^{(N)}=P+\G_qR_q\G_q^{-1}.
\end{equation}
The degree bounds on both final kernels identify every entry of these $q\times q$ matrices.

\begin{theorem}\label{thm:inverse}
For every $N\ge1$,
\begin{equation}\label{eq:inverse}
 \Omega_{2N}(\J_{2N}-C^{(N)})=I_{2N},\qquad
 \Omega_{2N}^{-1}=\J_{2N}-C^{(N)}.
\end{equation}
\end{theorem}
\begin{proof}
Abbreviate $G=\G_q$, $R=R_q$, $C=C^{(N)}$, $J=\J_q$. Equations~\eqref{eq:GC} and \eqref{eq:P-reflection} imply
$GC+RGCR=GRG^{-1}+RGRG^{-1}R-R$. Since $GJ=GRG^{-1}-R$, putting $X_*=J-C$ yields
\[
 GX_*+RGX_*R=-R.
\]
Also $RX_*R=-X_*$ by \eqref{eq:Creflection} and \eqref{eq:omega}. Multiplying this equality on the right by $-R$ gives
$(GR-RG)X_*=I$. Taking determinants in this square matrix identity proves that $\Omega_q$ is invertible, with the stated inverse.
\end{proof}

\section{Odd-dimensional lifting and central projection}\label{sec:odd}
Let $Q$ be the $(2N+1)\times2N$ matrix $Q_{ij}=\delta_{ij}+\delta_{i,j+1}$, and let $e_i=(-1)^i$. Then $Q^Te=0$ and $R_{2N+1}Q=QR_{2N}$. Four-term Pascal addition gives
\begin{equation}\label{eq:compression}
 Q^T\G_{2N+1}Q=\G_{2N},\qquad Q^T\Omega_{2N+1}Q=\Omega_{2N}.
\end{equation}
The inverse Gram matrix has generating polynomial
\begin{equation}\label{eq:Ginverse-poly}
 \sum_{i,j<q}(\G_q^{-1})_{ij}x^iy^j=\sum_{k=0}^{q-1}(1+x)^k(1+y)^k.
\end{equation}
Write $\G_q=DP_qP_q^TD$, where $D_{ii}=(-1)^i$ and $(P_q)_{ij}=\binom ij$. Inverting the two Pascal factors gives $(\G_q^{-1})_{ij}=\sum_{k<q}\binom ki\binom kj$, which gives \eqref{eq:Ginverse-poly}. Multiplication by $(1+x)(1+y)$ shifts the summation index, so
\begin{equation}\label{eq:G-lift}
 \G_{2N+1}^{-1}=Q\G_{2N}^{-1}Q^T+e_0e_0^T.
\end{equation}
Consequently
\[
 \J_{2N+1}=Q\J_{2N}Q^T+e_{2N}e_0^T-e_0e_{2N}^T.
\]
The last two terms are the coefficient matrix of $x^{2N}-y^{2N}$, the correction in \eqref{eq:odd-kernel}. Therefore
\begin{equation}\label{eq:gammaOdd}
 \Gamma_o:=\J_{2N+1}-O^{(N)}=Q\Omega_{2N}^{-1}Q^T.
\end{equation}

Put $p=h_{N+1}/b$. The polynomial defining $\nu$ satisfies
\begin{equation}\label{eq:nu-transform}
 \frac1{1+x}\mathscr V_N(\sigma x)
 =(-1)^Np(-x)^2-
 \frac{x^{2N+1}}{(1+x)^{2N+1}}p(1+x)p(-x).
\end{equation}
By \eqref{eq:uv}, reciprocity and \eqref{eq:sigma}, the left side is
$(-1)^Nu^*(x)p(-x)/(b(1+x)^{2N+1})$; insert the first identity in \eqref{eq:theta}. Since $u(\theta x)=-(1+x)bp(-x)$, this gives the displayed two terms.
Truncating at degree $2N+1$ and applying \eqref{eq:G-action} gives $\G_{2N+1}\nu=(-1)^N\operatorname{coeff}(p(-x)^2)$. The latter vector is reversal invariant, because $p$ is reciprocal and the degree is $2N$. So is $\nu$, and hence
\begin{equation}\label{eq:nu-null}
 \Omega_{2N+1}\nu=0,\qquad e^T\nu=\mathscr V_N(-1)=(-1)^N\ne0.
\end{equation}
The last evaluation follows from $h_{N+1}(0)=b$ and the reversed polynomial $u^*(-1)=(-1)^Nb$.
The compression in \eqref{eq:compression} has rank $2N$, while an odd skew matrix is singular. Thus $\rank\Omega_{2N+1}=2N$ and its kernel is $\Q\nu$.

Equation~\eqref{eq:gammaOdd} implies $Q^T\Omega_{2N+1}\Gamma_o=Q^T$ and $\Gamma_oe=0$. Hence $I-\Omega_{2N+1}\Gamma_o$ has image in $\ker Q^T=\Q e$. Multiplication on the left by $\nu^T$, using skew symmetry and \eqref{eq:nu-null}, determines its coefficient and gives
\begin{equation}\label{eq:rank-one-defect}
 \Omega_{2N+1}\Gamma_o=I-\frac{e\nu^T}{e^T\nu}.
\end{equation}
Let $c=N$ be the central coordinate and $\alpha=\nu_c>0$. Set
\[
 P_c=I-\frac{\nu e_c^T}{\alpha},\qquad \widehat\Gamma=P_c\Gamma_oP_c^T.
\]
The central row and column of $\widehat\Gamma$ are zero. Since $\Omega P_c=\Omega$ and $\nu^TP_c^T=0$, \eqref{eq:rank-one-defect} yields
\[
 \Omega_{2N+1}\widehat\Gamma=I-\frac{e_c\nu^T}{\alpha}.
\]
For noncentral row and column indices the central summand in this product is zero. If a superscript $\circ$ means deletion of the center, it follows that
\begin{equation}\label{eq:deleted-inverse}
 \Omega_{2N+1}^{\circ}\widehat\Gamma^{\circ}=I_{2N},\qquad
 (\Omega_{2N+1}^{\circ})^{-1}=\widehat\Gamma^{\circ}.
\end{equation}
This is an inverse after projection and deletion, not an identification of an inverse principal block with the inverse of a principal block.

\section{The even and odd matrix bridges}\label{sec:connections}
\subsection{Even dimension}
Use reflection-paired row coordinates and write
\[
 G=\begin{pmatrix}A&B\\B^T&D\end{pmatrix},\qquad X=D-A+B-B^T.
\]
Here and in this section block letters are local to the displayed matrices. Split the columns of $\B_{2N}$ in their original order into two groups of size $N$; its paired-row form is
\[
 \B=\begin{pmatrix}0&V\\H_0&W\end{pmatrix},\qquad
 V=R_N\U_N,\qquad H=H_0H_0^T.
\]
The entries of $\B_q$ determine these blocks. The zero block follows from $\binom i{q-1-j}=0$ for $i,j<N$. The top-right entries are $(-1)^{i+k-N+1}\binom i{N-1-k}$, $0\le i,k<N$, which are $(R_N\U_N)_{ik}$. The remaining square block $H_0$ is triangular after reversal, with diagonal entries of absolute value one. Thus $V,H_0$ are invertible and $\det V,\det H_0\in\{1,-1\}$.
Put $L_e=K_{2N}-E_{2N,N}$. Equation~\eqref{eq:K-congruence} and $\B\B^T=G$ give
\begin{equation}\label{eq:even-elim}
 \B L_e\B^T=\begin{pmatrix}X&-X\\X&H-X\end{pmatrix}.
\end{equation}
In fact, $(R+I)G(R-I)$ has blocks
$\begin{pmatrix}X&-X\\X&-X\end{pmatrix}$ in these coordinates, and $\B(I-E_{2N,N})\B^T$ contributes only $H$ in the lower-right block.
Also
$\Omega_{2N}\binom I{-I}=\binom II X$.
If $Xa=0$, the left side annihilates $(a,-a)$; Theorem~\ref{thm:inverse} then gives $a=0$. Thus $X$ is invertible. Subtracting the first block row in \eqref{eq:even-elim} from the second gives an upper triangular block matrix with diagonal $X,H$, so $L_e$ is invertible. Its transformed inverse is
\begin{equation}\label{eq:even-block-inverse}
 \begin{pmatrix}X^{-1}-H^{-1}&H^{-1}\\-H^{-1}&H^{-1}\end{pmatrix}.
\end{equation}
A direct block multiplication with \eqref{eq:even-elim} verifies the inverse.

For a reversal-antisymmetric skew matrix define its fold as minus the sum of its left-left and left-right blocks. The relation $\Omega\binom I{-I}=\binom II X$ implies that the fold of $\Omega^{-1}$ is $-X^{-1}$. The fold of $C^{(N)}$ is $D_N$. If $D_0$ is the fold of $J$, Theorem~\ref{thm:inverse} therefore gives
\begin{equation}\label{eq:Dfold}D_N=D_0+X^{-1}.\end{equation}
The inverse Gram matrix, obtained by inverting the displayed $\B$, is
\[
 G^{-1}=\begin{pmatrix}
 V^{-T}(I+W^TH^{-1}W)V^{-1}&-V^{-T}W^TH^{-1}\\
 -H^{-1}WV^{-1}&H^{-1}
 \end{pmatrix}.
\]
Compute $J=RG^{-1}-G^{-1}R$ block by block and fold; this gives
\begin{equation}\label{eq:D0}
 V^TD_0V=I+W^TH^{-1}W-V^TH^{-1}V-W^TH^{-1}V+V^TH^{-1}W.
\end{equation}
Writing the upper blocks of $G^{-1}$ as $A_0,B_0$ and its lower-right block as $H^{-1}$, we have $J_{LL}=B_0^T-B_0$, $J_{LR}=H^{-1}-A_0$. Negating their sum and conjugating by $V$ yields \eqref{eq:D0}.
The original-coordinate tail of $L_e^{-1}$ is the congruence of \eqref{eq:even-block-inverse} by $\binom VW$, namely
\[
 V^TX^{-1}V-V^TH^{-1}V+V^TH^{-1}W-W^TH^{-1}V+W^TH^{-1}W.
\]
Add $I$ and use \eqref{eq:Dfold}--\eqref{eq:D0}. We have proved the even connection
\begin{equation}\label{eq:even-connection}
 I_N+(L_e^{-1})_{N:2N,N:2N}=\U_N^TR_ND_NR_N\U_N.
\end{equation}

\subsection{Odd dimension}
Order rows as left half, reversed right half, center, and write
\[
 G=\begin{pmatrix}A&B&a\\B^T&D&d\\a^T&d^T&g\end{pmatrix},\quad
 X=D-A+B-B^T,\quad w=2(d-a),\quad \nu=\begin{pmatrix}v\\v\\\alpha\end{pmatrix}.
\]
The deleted matrix in \eqref{eq:deleted-inverse} is invertible, so the preceding argument proves $X$ invertible. Since
$\Omega\begin{pmatrix}I\\-I\\0\end{pmatrix}=\begin{pmatrix}X\\X\\w^T\end{pmatrix}$ and $\nu^T\Omega=0$,
\begin{equation}\label{eq:wXi}w^TX^{-1}=-2v^T/\alpha.\end{equation}
In this order
\[
 O^{(N)}=\begin{pmatrix}A_o&B_o&h\\-B_o&-A_o&-h\\-h^T&h^T&0\end{pmatrix}.
\]
For noncentral $i,j$, central projection changes an entry to
$O_{ij}-\nu_iO_{cj}/\alpha-O_{ic}\nu_j/\alpha$. Folding cancels the $vh^T$ terms and gives
\begin{equation}\label{eq:So}
 S_o=-A_o-B_o+2hv^T/\alpha.
\end{equation}
This is also the Schur complement of $\alpha$ in $\EE_N$. If $S_0$ is the same projected fold of $J$, equations~\eqref{eq:gammaOdd} and \eqref{eq:deleted-inverse} imply
\begin{equation}\label{eq:So-fold}S_o=S_0+X^{-1}.\end{equation}

Split the original columns of $\B_{2N+1}$ into sizes $N+1,N$, and group the last $N+1$ paired rows as reversed right, center. Then
\[
 \B=\begin{pmatrix}0&V\\H_0&Y\end{pmatrix},\quad
 V=R_N\U_N,\quad H=H_0H_0^T,\quad E=(I_N\ \ 0),\quad Z=\binom X{w^T}.
\]
The same entry formula for $\B_q$ proves the zero block and $V$, and gives $\det H_0=\pm1$ (reverse its rows to obtain a triangular Pascal block). Define $L_o=K_{2N+1}-E_{2N+1,N}$. In these blocks reversal is $\begin{pmatrix}0&E\\E^T&e_*e_*^T\end{pmatrix}$, where $e_*$ is the last coordinate in the lower block. Substitution in \eqref{eq:K-congruence} gives
\begin{equation}\label{eq:odd-elim}
 \B L_o\B^T=\begin{pmatrix}X&-XE\\Z&H-ZE\end{pmatrix}.
\end{equation}
The upper blocks of $(R+I)G(R-I)$ are $X,-XE$, the lower-left block is $Z$, and the lower-right block is $-ZE$; the remaining Gram contribution is $H$. Subtracting $ZX^{-1}$ times the first row of blocks from the second leaves $H$, and the inverse is
\begin{equation}\label{eq:odd-inverse-block}
 \begin{pmatrix}
 X^{-1}-EH^{-1}ZX^{-1}&EH^{-1}\\
 -H^{-1}ZX^{-1}&H^{-1}
 \end{pmatrix}.
\end{equation}
Multiplication by \eqref{eq:odd-elim} verifies the inverse.
Equation~\eqref{eq:wXi} says
\begin{equation}\label{eq:ZXi}ZX^{-1}=E^T-e_*2v^T/\alpha.\end{equation}
Let $F=YV^{-1}$ and $Z_0=H^{-1}$. The blocks of $G^{-1}$ are
\[
 G^{-1}=\begin{pmatrix}A_0&-F^TZ_0\\-Z_0F&Z_0\end{pmatrix},\qquad
 A_0=V^{-T}V^{-1}+F^TZ_0F.
\]
Thus
\[
 J_{LL}=-EZ_0F+F^TZ_0E^T,\quad J_{LR}=EZ_0E^T-A_0,\quad
 J_{Lc}=(E+F^T)Z_0e_*.
\]
The projected fold, including its central correction $2J_{Lc}v^T/\alpha$, is consequently
\begin{equation}\label{eq:S0}
 S_0=V^{-T}V^{-1}+(E+F^T)H^{-1}(F-E^T+e_*2v^T/\alpha).
\end{equation}
The original-coordinate tail of $L_o^{-1}$ is the congruence of \eqref{eq:odd-inverse-block} by $\binom VY$. Multiplying on the left by $V^{-T}$ and on the right by $V^{-1}$ expands it as
\[
 X^{-1}+EH^{-1}F+F^TH^{-1}F-(E+F^T)H^{-1}ZX^{-1}.
\]
After adding $V^{-T}V^{-1}$ and inserting \eqref{eq:ZXi}, this is $X^{-1}+S_0=S_o$. Hence the odd connection is
\begin{equation}\label{eq:odd-connection}
 I_N+(L_o^{-1})_{N+1:2N+1,N+1:2N+1}=\U_N^TR_NS_oR_N\U_N.
\end{equation}
All tails in \eqref{eq:even-connection} and \eqref{eq:odd-connection} are taken only after returning to the original candidate coordinates.

\section{Normalization, the main theorem, and boundary cases}\label{sec:closure}
\subsection{Normalization from unrestricted enumeration alone}
Set $s=0$ in Theorem~\ref{thm:true}. Since $B_{n,0}=A_n$ by definition,
\begin{equation}\label{eq:normalization}
 A_{2N}=A_N^2\det D_N,\qquad
 A_{2N+1}=A_N^2\det\EE_N=A_N^2\alpha\det S_o.
\end{equation}
The normalization thus uses the unrestricted count in Theorem~\ref{thm:true}.
For a coordinate embedding $\Pi$, the finite determinant lemma gives
\begin{equation}\label{eq:det-add}
 \det(L+\Pi\Pi^T)=\det L\det(I+\Pi^TL^{-1}\Pi).
\end{equation}
Apply it to $L_e,L_o$, adding back all $N$ missing diagonal entries, and use the two connections. Since $\det\U_N=1$ and $(\det R_N)^2=1$,
\[
 \det K_{2N}=\det L_e\det D_N,\qquad
 \det K_{2N+1}=\det L_o\det S_o.
\]
Equations~\eqref{eq:detK} and \eqref{eq:normalization}, with $A_n>0$ and $\alpha>0$, justify cancellation and give
\begin{equation}\label{eq:base-det}
 \det L_e=A_N^2,\qquad \det L_o=A_N^2\alpha.
\end{equation}

\subsection{Adding back an arbitrary initial segment}
Fix $0\le s\le N$ and $m=N-s$. Adding $1$ to the diagonal entries $N,\ldots,N+m-1$ of $L_e$ gives $K_{2N}-E_{2N,s}$. In odd dimension the corresponding indices are $N+1,\ldots,N+m$. For any $N\times N$ matrix $Z$, upper triangularity gives
\begin{equation}\label{eq:front-congruence}
 (\U_N^TZ\U_N)_{0:m,0:m}
 =(\U_N)_{0:m,0:m}^TZ_{0:m,0:m}(\U_N)_{0:m,0:m}.
\end{equation}
In fact in column $j<m$ the entry $(\U_N)_{ij}$ vanishes for $i>j$, so every contributing summation index is less than $m$. The leading determinant is therefore unchanged. This assertion is only about leading blocks, not arbitrary principal blocks.
Using $Z=R_ND_NR_N$ or $R_NS_oR_N$, simultaneous row and column reversal changes the leading interval into $s:N$ without altering the determinant. Equations~\eqref{eq:det-add}, \eqref{eq:front-congruence} and the connections first give the purely finite identities
\begin{equation}\label{eq:relative-bridges}
 \begin{split}
 \det(K_{2N}-E_{2N,s})&=\det L_e\det(D_N)_{s:N,s:N},\\
 \det(K_{2N+1}-E_{2N+1,s})&=\frac{\det L_o}{\alpha}
 \det(\EE_N)_{s:N+1,s:N+1}.
 \end{split}
\end{equation}
For the second equality, the central scalar in every indicated tail of $\EE_N$ is the same nonzero $\alpha$, and its Schur complement is $(S_o)_{s:N,s:N}$. Finally \eqref{eq:base-det} and Theorem~\ref{thm:true} prove
\begin{align}
 \det(K_{2N}-E_{2N,s})&=A_N^2\det(D_N)_{s:N,s:N}=B_{2N,s},\label{eq:even-final}\\
 \det(K_{2N+1}-E_{2N+1,s})&=A_N^2\det(\EE_N)_{s:N+1,s:N+1}=B_{2N+1,s}.\label{eq:odd-final}
\end{align}
Proposition~\ref{prop:candidate} at $\lambda=1$ proves Theorem~\ref{thm:main} in the nontrivial range.

\subsection{Endpoints}
For $n=1$ the only ASM is $[1]$, so $B_{1,0}=1$ and $B_{1,1}=0$. At $s=1$, $h_1=b_1=1$, $f_0^+=(1-z)/z$, $f_0^-=(1+z)/z$, and \eqref{eq:cpM} gives $M_{00}=1$. No zero-dimensional commutator inverse is used.
For every $n$, $s=0$ has $M$ empty and $B_{n,0}=A_n$, agreeing with \eqref{eq:detK}.
At $2s=n$, write $n=2N$, $s=N$. The known integral has zero variables and gives $A_N^2$ directly; the determinant of the empty tail of $D_N$ is $1$. Its candidate counterpart is $\det L_e=A_N^2$.
At $2s=n-1$, write $n=2N+1$, $s=N$. There is one residue, giving $A_N^2\nu_N$, and the tail of $\EE_N$ is the $1\times1$ matrix $[\nu_N]$. The candidate counterpart is $\det L_o=A_N^2\nu_N$.
For $2s>n$, the combinatorial and rank arguments in Section~\ref{sec:candidate} make both sides zero, completing the stated parameter range.

\section{Formal verification}\label{sec:lean}
The finite-dimensional algebraic core has been formalized over $\Q$ in Lean 4.19.0 with Mathlib commit
\texttt{c44e0c8ee63ca166450922a373c7409c5d26b00b}. It covers the refined-polynomial identities, kernel divisibility and truncation, Theorem~\ref{thm:inverse}, the odd nullspace and central projection, both matrix connections, and the relative bridges \eqref{eq:relative-bridges} for arbitrary admissible parameters.

The absolute factor $A_N^2$ in the formal statements requires normalization hypotheses specifying $\det K_n=A_n$ and \eqref{eq:normalization}; the relative bridges and core inverse are independent of these hypotheses. ASM combinatorial definitions, the CP24 integral input, the analytic and Pfaffian derivation of the enumeration formulas, and the candidate-to-CP/FR identification remain outside the formalization. The external asymptotic theorem used in Corollary~\ref{cor:tw} is also outside its scope.

The formalized theorems use only the standard logical axioms reported by Mathlib. The Lean source code, build instructions and verification details accompany the submission as separate supplementary files. Supplementary Material, Section~S.2, records the scope and dependencies.

\appendix
\section{Coefficient ratios for the recurrence}\label{app:certificate}
For $m\ge5$, $0\le a\le\lfloor(m+1)/2\rfloor$, divide the coefficients in \eqref{eq:normalized-rec} by $c_{m,a}$. In order, the ratios for
$c_{m,a-1},c_{m,a-2},c_{m-1,a},c_{m-1,a-1},c_{m-1,a-2},c_{m-1,a-3},c_{m-2,a-2}$ are
\begin{align*}
 u_1&=\frac{a(2m-a-1)}{(m+a-1)(m-a)},\\
 u_2&=\frac{a(a-1)(2m-a)(2m-a-1)}{(m+a-2)(m+a-1)(m-a)(m-a+1)},\\
 v_0&=\frac{(m-a-1)(2m-2)(2m-3)}{(m+a-1)(2m-a-2)(2m-a-3)},\\
 v_1&=\frac{a(2m-2)(2m-3)}{(m+a-1)(m+a-2)(2m-a-2)},\\
 v_2&=\frac{a(a-1)(2m-2)(2m-3)}{(m+a-1)(m+a-2)(m+a-3)(m-a)},\\
 v_3&=v_2\frac{(a-2)(2m-a-1)}{(m+a-4)(m-a+1)},\\
 w_2&=\frac{a(a-1)(2m-2)(2m-3)(2m-4)(2m-5)}
 {(m+a-4)(m+a-3)(m+a-2)(m+a-1)(2m-a-3)(2m-a-2)}.
\end{align*}
Each denominator is nonzero: $m+a-4\ge1$, $m-a\ge2$ and $2m-a-3\ge4$ in this range, and the other factors are larger positive integers. The falling products in the numerators give the zero values of the missing shifted coefficients at $a=0,1,2$. The coefficient assertion is
\[
 1-2u_1+u_2-v_0+\tfrac32v_1+\tfrac32v_2-v_3-\beta_{m-1}w_2=0.
\]
Multiply its eight terms by
$D=(m-a)(m-a+1)(2m-a-3)(2m-a-2)\prod_{j=1}^4(a+m-j)$.
Clearing this denominator gives an identity in $\mathbb Z[a,m]$. The eight resulting polynomials $N_0,\ldots,N_7$ and the certificate $\sum_{i=0}^7N_i=0$ are displayed in Supplementary Material, Section~S.1 (\texttt{CP\_SUPPLEMENTARY\_MATERIAL.tex}). The same identity is verified by \texttt{CP.recurrence\_certificate} in \texttt{CP/Polynomial.lean}.
Dividing by the nonzero $D$ completes the coefficient calculation. Together with the low-order cases and reciprocity in Lemma~\ref{lem:recurrence}, this proves the full recurrence.

\section{Computational checks}
Historical exact integer and rational checks covered polynomial orders through $20$, candidate sizes through $16$, direct ASM enumeration through $12$, and both bridges for $1\le N\le12$. They served to detect sign and indexing errors and play no part in the general proof. These reported ranges have not been independently rerun; their provenance is recorded in Supplementary Material, Section~S.3.

\end{document}